\documentclass{article}
\usepackage[utf8]{inputenc}
\usepackage{listings}
\usepackage{amsmath}
\usepackage{amsthm}
\usepackage{amsfonts}
\usepackage{amssymb}
\usepackage[english]{babel}
\usepackage[a4paper]{geometry}
\usepackage{graphicx}
\usepackage{color,psfrag}
\usepackage{fancyhdr}
\usepackage{cite}
\usepackage{makeidx}
\usepackage{enumerate}
\usepackage{mathpazo}
\usepackage{mathrsfs}
\usepackage{float}
\usepackage{datetime}
\usepackage{titlesec}
\usepackage[labelfont=bf,textfont=it,labelsep=period,width=.8\textwidth]{caption}
\usepackage[unicode,pdfstartview={XYZ null null 1},pdfview={XYZ null null 1},pdfpagemode=UseOutlines]{hyperref}
\DeclareFontFamily{U}{wncy}{}
\DeclareFontShape{U}{wncy}{m}{n}{<->wncyr10}{}
\DeclareSymbolFont{mcy}{U}{wncy}{m}{n}
\DeclareMathSymbol{\Sh}{\mathord}{mcy}{"58}

\newcommand{\A}{{\cal A}}

\newcommand{\R}{\mathbb R}
\newcommand{\Q}{\mathbb Q}
\newdate{date}{10}{03}{2017}

\newcommand{\F}{\mathbb F}

\newtheorem{theorem}{Theorem}
\newtheorem{fact}[theorem]{Fact}

\newtheorem{corollary}[theorem]{Corollary}
\newtheorem{proposition}[theorem]{Proposition}
\newtheorem{definition}[theorem]{Definition}

\newtheorem{problem}{Problem}
\theoremstyle{remark}\newtheorem{example}[theorem]{Example}
\theoremstyle{remark}\newtheorem{remark}[theorem]{Remark}

\title{Algebraic conditions for additive functions over the reals and over finite fields}
\author{\normalsize
 \begin{minipage}{0.3\linewidth}
    \large
    P\'eter Kutas \\
    \footnotesize
Institute for Computer Science and Control,
Hungarian Acad. Sci.
    \texttt{kutasp@sztaki.hu} \\
    \normalsize
  \end{minipage}
  }

\begin{document}

\maketitle
\begin{abstract}
Let $C$ be an affine plane curve. We consider additive functions $f: K\rightarrow K$ for which $f(x)f(y)=0$, whenever $(x,y)\in C$. We show that if $K=\mathbb{R}$ and $C$ is the hyperbola with defining equation $xy=1$, then there exist nonzero additive functions with this property. Moreover, we show that such a nonzero $f$ exists for a field $K$ if and only if $K$ is transcendental over $\Q$ or over $\F_p$, the finite field with $p$ elements. We also consider the general question when $K$ is a finite field. We show that if the degree of the curve $C$ is large enough compared to the characteristic of $K$, then $f$ must be identically zero. 
\end{abstract}
\bigskip


\section{Introduction}

Let $K$ be field. A function $f:K\rightarrow K$ is additive if 
$$f(x+y)=f(x)+f(y)$$
holds for every $x,y\in K$. If the characteristic of $K$ is zero then an additive function is also $\Q$-linear \cite[Theorem 5.2.1]{Kuczma}. 
For further information on additive functions (and functional equations in general) the reader is referred to \cite{Kuczma}. 

It is well-known that if $K=\mathbb{R}$ then every continuous additive function $f:\mathbb{R}\rightarrow \mathbb{R}$ is of the form $f(x)=f(1)x$. However, there exist non-continuous additive functions which behave quite irregularly. Most prominently their graph is dense in the plane. Many similar theorems can be found in \cite{Kuczma}, however most are of analytic flavour. 

The following more algebraic question was posed by Gy. Szab\'o \cite{Szabo} (motivated by a question of W. Benz \cite{Benz2}): Let $C: x^2+y^2=1$ be the unit circle and suppose that $f:\mathbb{R}\rightarrow \mathbb{R}$ is additive and $f(x)f(y)=0$ whenever $(x,y)\in C$. Does this imply that $f$ is identically zero? The solution was published in a paper \cite{KRS} by Z. Kominek, L. Reich and J. Schwaiger, where the authors prove that the implication is true. In the same paper the authors also consider different curves (instead of the unit circle), curves of the form $y=p(x)$, where $p$ is a polynomial with rational coefficients, and the hyperbola $x^2-y^2=1$. In the paper \cite{BFK} the curves $x^2-ny^2=1$ (where $n$ is a square-free positive integer) are studied. In all cases, the implication that $f$ is identically zero turned out to be true. Moreover, Z. Boros and W. Fechner generalized the unit circle problem to generalized polynomials \cite{BF} and in \cite{BFK} a stability version of the problem is examined. In \cite{BFK} the case when the curve is $C:xy=1$ was left open. These previous results suggested that such an $f$ must be identically zero as well.

Another motivation to study this problem comes from the following results. Let $f$ be an additive function such that $f(x)f(\frac{1}{x})=b$, where $b$ is fixed. If $b$ is negative, then no such function exists. If $b$ is positive, then one can show that $f$ must be continuous \cite[Chapter I]{Kannappan}, thus only the case $b=0$ was not known before. Interestingly, there exists a non-continuous additive $f$, for which $f(x)f(\frac{1}{x})>0$ for every nonzero $x$ \cite{Benz},\cite{Bergman}. 

One of the main results of this paper is that surprisingly there exists an additive $f:\mathbb{R}\rightarrow\mathbb{R}$ such that $f(x)f(\frac{1}{x})=0$ (when $x\neq 0$) which is not identically zero on $\mathbb{R}$. The example uses the theory of valuations of fields. Moreover, with this technique, we are able to construct a family of curves $C$, for which the above implication is false. 

We also consider the general problem over finite fields. Let $C$ be an absolutely irreducible smooth curve and let $f:\mathbb{F}_{p^k}\rightarrow \mathbb{F}_{p^k}$ be an additive function such that $f(x)f(y)=0$ whenever $(x,y)\in C$. In this setting we are able to prove more general results. Using the classical Hasse-Weil bound on the number of points of a smooth absolutely irreducible curve defined over a finite field, we are able to show that if the degree of the curve is small compared to the characteristic of the field, then $f$ must be identically zero. Our approach can be applied to higher dimensional varieties as well. 

The structure of the paper is the following. Section 2 is devoted to preliminaries, it does not contain any new results. We recall basic facts of valuation theory and the Hasse-Weil theorem. Section 3 is devoted to the finite field case, whilst in Section 4 we study the problem over the real numbers.

\section{Preliminaries}

This section is divided into two subsections. In the first subsection we define some notions and state some theorems concerning algebraic curves over finite fields. In the second subsection, we introduce the concept of valuations of fields.

\subsection{Curves over finite fields}

The theory of algebraic curves is a vast topic, therefore this subsection is not meant to be a thorough introduction (for more details the reader is referred to \cite{Fulton}). Our goal is to establish a bound on the number of points over a finite field, of a smooth projective curve. 

\begin{definition}
Let $K$ be a field. An affine (resp. projective) plane curve is an affine (resp. projective) variety of dimension one in $\overline{K}^2$ (resp $\mathbb{P}\overline{K}^2$). Here $\overline{K}$ denotes the algebraic closure of $K$, and $\mathbb{P}\overline{K}^2$ denotes the projective plane over $\overline{K}$.
\end{definition}
\begin{remark}
Informally one can think of the following. An affine plane curve (defined over $K$) is the zero set of an irreducible polynomial $f\in K[x,y]$ in $\overline{K}^2$. A projective plane curve is the zero set of an irreducible homogeneous polynomial $g\in K[x,y,z]$ in $\mathbb{P}\overline{K}^2$, where $\mathbb{P}\overline{K}^2$ denotes the projective plane (over $\overline{K}$). 
\end{remark}

In this paper we will consider affine curves, however certain theorems are stated in terms of projective curves. 

Our main goal is to give a bound on the number of points of a smooth, absolutely irreducible curve. We do not define these notions here, as we will not use their definition. It is enough to note that this is a large (and the most interesting) class of affine curves. For further details the reader is referred to \cite{Fulton}. The bound contains a quantity called the genus of the curve. We do not define the genus here (as it is quite complicated and we do not use it later), the reader is referred to \cite[Section 8.3]{Fulton}. We recall the well-known Hasse-Weil theorem \cite{HW}:
\begin{fact}[Hasse-Weil]\label{HW}
Let $p$ be a prime number and let $C$ be a smooth absolutely irreducible projective curve defined over the finite field $\F_{p^k}$. Let $N(C)$ be the number of points of $C$ in $\F_{p^k}$ and let $g$ be the genus of the curve. Then the following inequality holds:
$$ |N(C)-p^k-1|\leq 2gp^{\frac{k}{2}}.$$
\end{fact}

Fortunately, the genus of the curve can be bounded using the degree of the curve (the degree of its defining polynomial) \cite[Section 8.3, Corollary 1]{Engler}:
\begin{fact}
Let $C$ be a smooth projective curve of degree $d$. Than its genus $g$ satisfies the inequality  $$g\leq \frac{(d-1)(d-2)}{2}.$$
\end{fact}

We conclude the subsection by a fact (which is a consequence of B\'ezout's theorem \cite[Section 5.3]{Fulton}):
\begin{fact}\label{Bezout}
Let $C_1$ and $C_2$ be projective curves defined over a field $K$ of degree $d_1$ and $d_2$ respectively. Then $C_1$ and $C_2$ intersect in at most $d_1d_2$ points.
\end{fact}
\begin{remark}
Actually counting multiplicities they intersect in exactly $d_1d_2$ points over algebraically closed fields.
\end{remark}

This fact will be useful for us as we later need a bound on the number of affine points of a smooth curve. 

\subsection{Valuations}

This subsection is based on \cite[Chapter 2 and 3]{Engler}. We define the notion of valuations of a field and state a theorem of Chevalley, about the extensions of valuations. First we recall the notion of an ordered abelian group.

\begin{definition}
An abelian group $(\Gamma,+,0)$ together with a binary relation $\leq$ is called {\em ordered} if the following conditions hold for all $\gamma, \delta, \lambda\in\Gamma$:
\begin{enumerate}
\item $\gamma\leq \gamma$,
\item if $\gamma\leq \delta$ and $\delta\leq \gamma$ then $\gamma=\delta$,
\item if $\gamma\leq \delta$ and $\delta\leq \lambda$ then $\gamma\leq \lambda$,
\item either $\gamma\leq \delta$ or $\delta\leq \gamma$,
\item $\gamma\leq \delta$ implies that $\gamma+\lambda\leq \delta+\lambda$. 
\end{enumerate}
\end{definition}
\begin{example}
The real numbers with respect to addition form an ordered abelian group. 
\end{example}

\begin{definition}
Let $K$ be a field and let $\Gamma$ be an ordered abelian group. Let $\infty$ be a symbol for which $\infty\geq x$ holds for every $x\in\Gamma$. Then a valuation on $K$ is a function $v:K\rightarrow \Gamma\cup \{\infty\}$ with the following properties:

\begin{enumerate}
    \item $v(x)=\infty$ if and only if $x=0$,
    \item $v(xy)=v(x)+v(y)$,
    \item $v(x+y)\geq min(v(x),v(y))$.
\end{enumerate}
\end{definition}
\begin{example}\label{ex}
\begin{itemize}
    \item Every field has the trivial valuation, where $v(x)=0$, whenever $x\neq 0$.
    \item Let $K=\Q(t)$, the field of rational functions over $\Q$. Let $f,g\in \Q[t]$. Then 
    $$v(\frac{f}{g})=deg(g)-deg(f)$$
    is a valuation (the degree of the zero polynomial is $-\infty$).
    \item Let $K=\Q$ and let $p$ be a prime number. Then every rational number $\frac{a}{b}$ can be written in a form $p^k\frac{a'}{b'}$, where $a'$ and $b'$ are no longer divisible by $p$. Then $v(\frac{a}{b})=k$ is a valuation on $\Q$, called the $p$-adic valuation.
\end{itemize}
\end{example}

Valuations are a key concept in algebraic number theory. We only need a fraction of the theory. First we observe that elements with nonnegative valuation form a subring with a special property:
\begin{proposition}\label{val}
Let $K$ be a field and let $v$ be a valuation on $K$. Let $O=\{x\in K| v(x)\geq 0\}$. Then $O$ has the following two properties:
\begin{enumerate}
\item $O$ is a subring of $K$.
\item For every nonzero element $x\in K$, either $x$ or $\frac{1}{x}$ is in $O$.
\end{enumerate}
\end{proposition}
\begin{proof}
We have to show that $O$ is closed under addition and multiplication. Let $x,y\in O$. The second property of a valuation implies that $v(x+y)\geq min(v(x),v(y))\geq 0$, which proves that $O$ is closed under addition. By the third property of a valuation $v(xy)=v(x)+v(y)\geq 0$, hence $O$ is closed under multiplication. The second part of the statement follows again from the third property. 
\end{proof}

This motivates the following definition:
\begin{definition}
Let $K$ be a field. Then a subring $O$ is a valuation ring if for every nonzero element $x\in K$ either $x$ or $\frac{1}{x}$ is in $O$. 
\end{definition}

Valuation rings always correspond to valuations, i.e., if $O$ is a valuation ring of $K$ then there exists a valuation on $K$, such that $O$ consists of those elements whose valuation is nonnegative \cite[Proposition 2.1.2]{Engler}. Valuation rings have a unique maximal ideal (the elements whose valuation is positive). Now we state Chevalley's theorem \cite[Theorem 3.1.1.]{Engler}:

\begin{fact}[Chevalley] \label{Chevalley}
Let $K$ be a field and let $R$ be a subring of $K$ with a prime ideal $P$. Then there exists a valuation ring $O$ of $K$ with maximal ideal $M$ such that $R\subset O$ and $M\cap R=P$.
\end{fact}

A corollary of Fact \ref{Chevalley} is that a valuation on a smaller field can be extended to a larger field (however, the value group may grow):

\begin{fact}\label{ext}
Let $K$ be a field with a valuation $v$ and let $L$ be an extension of $K$. Then there exists a valuation $w$ on $L$ such that $w$ restricted to $K$ is $v$. 
\end{fact}

\section{Additive functions and curves over finite fields}

Throughout the section let $p$ be an odd prime number. We consider additive functions in finite fields with the following property. Let $C$ be a smooth absolutely irreducible curve defined over $F_{p^k}$. Let $f$ be an additive function on $\F_{p^k}$ such that $f(x)f(y)=0$ whenever $(x,y)\in C$. The question is whether there exists an $f$ with this property which is not identically zero. 
First we make some basic observations.
\begin{proposition}\label{basic}
Let $\F_{p^k}$ be the finite field with $p^k$ elements, where $p$ is an odd prime. Then an additive function $f:\mathbb{F}_{p^k}\rightarrow\mathbb{F}_{p^k}$ is also an $\F_p$-linear function. 
\end{proposition}
\begin{proof}
We have to show that $f(\lambda x)=\lambda f(x)$, whenever $\lambda\in F_p$. The elements of $\F_p$ can be represented by the residues modulo $p$, i.e., by the set $\{0,1,\dots, p-1\}$. We prove our claim by induction. By the additive property of $f$ we have that $f(0)=f(0)+f(0)$, hence $f(0)=0$. Thus for $\lambda=0$ or $\lambda=1$ we are done. Assume the claim is true for $\lambda$, we show it is true for $\lambda+1$. 
$$ f((\lambda+1)x)=f(\lambda x+x)=f(\lambda x)+f(x)=\lambda f(x)+f(x)=(\lambda+1)f(x).$$
Thus by induction we are done.
\end{proof}
\begin{remark}\label{subspace}
An easy consequence of this is that the set $H=\{x\in \F_{p^k}|~f(x)=0\}$ is an $\F_p$-subspace.
\end{remark}

Before we state our main theorem we give a short remark on the number of affine points of a curve.
\begin{remark}\label{affine}
Assume the curve $C$ has degree $d$. If one is interested in the number of affine points then one has to subtract $d$ from the lower bound. Indeed, projective points lie on a line, and a line and a degree $d$ curve intersect in at most $d$ points by Fact \ref{Bezout}.  
\end{remark}

Now we are ready to state the main theorem of this section:
\begin{theorem}\label{main}
Let $C$ be a smooth absolutely irreducible affine curve of degree $d$ defined over $\F_{p^k}$. Assume that the following inequality holds:
\begin{equation}\label{number}
\frac{p^k+1-(d-1)(d-2)p^{\frac{k}{2}}-d}{d}>2p^{k-1}.
\end{equation}
Then an additive function $f:\mathbb{F}_{p^k}\rightarrow \mathbb{F}_{p^k}$ for which $f(x)f(y)=0$ whenever $(x,y)\in C$, is identically zero.
\end{theorem}
\begin{proof}
Consider the following bipartite graph $G$. The two independent sets $S$ and $T$ are each labelled with the elements of the finite field $\F_{p^k}$. An edge goes from a vertex $x$ to the vertex $y$ in the opposite set if and only if the point $(x,y)$ lies on the curve $C$. Observe that the number of edges of $G$ is just the number of affine points of $C$. Since $C$ has degree $d$, every vertex has degree at most $d$. The vertices of $G$ corresponding to elements $x\in F_{p^k}$ for which $f(x)=0$ must cover all the edges of $G$, otherwise there would be an $x$ and a $y$ such that $f(x)f(y)\neq 0$. Let $m$ be the number of edges of $G$. Since the degree of every vertex is at most $d$, one needs at least $m/d$ vertices for covering all the edges of $G$. As we have noted, the vertices corresponding to elements $x\in\F_{p^k}$ for which $f(x)=0$ cover all the edges. The number of this vertices is exactly two times the cardinality of the set $H=\{x\in \F_{p^k}| f(x)=0\}$. Inequality \ref{number} and Fact \ref{HW} together imply (considering Remark \ref{affine}) that the cardinality of $H$ is larger than $p^{k-1}$. Remark \ref{subspace} states that $H$ is an $\F_p$-subspace. Since the only $\F_p$-subspace of $\F_{p^k}$ consisting of more than $p^{k-1}$ elements is $\F_{p^k}$ itself, $f$ must be identically zero. 
\end{proof}

\begin{remark}\label{two}
We conjecture that the 2 on the right hand side of Inequality \ref{number} can be omitted by a more clever construction of the graph $G$. 
\end{remark}

The theorem says that if the characteristic of the finite field is large enough compared to the degree of the curve then $f$ must identically zero. For two interesting classes of curves (conics and elliptic curves) we apply Theorem \ref{main} to provide explicit bounds. For the theory of elliptic curves the reader is referred to \cite{Silverman}.
\begin{corollary}
Let $f:\mathbb{F}_{p^k}\rightarrow \mathbb{F}_{p^k}$ be an additive function and let $C$ be a smooth curve. Assume that $f(x)f(y)=0$ whenever $(x,y)\in C$.
\begin{enumerate}
\item If $C$ is a conic and $p\geq 5$ then $f$ must be identically zero. 
\item Let $C$ be an elliptic curve. Then $f$ must be identically zero if one of the following holds:
\begin{itemize}
    \item $p>13$,
    \item $p=7$ and $k>2$,
    \item $p=11$ or $p=13$ and $k>1$.
\end{itemize}
\end{enumerate}
\end{corollary}

\begin{proof}
If $C$ is a conic than it has degree 2 and at least $p^k-1$ points. Thus we have to examine when $\frac{p^k-1}{2}>2p^{k-1}$ holds. This is equivalent to the inequality 
$$p^k-1>4p^{k-1}$$
which holds if $p\geq 5$.

If $C$ is an elliptic curve then we note first that it always has exactly 1 point at infinity. Thus we have to examine when the inequality $\frac{p^k-2p^{\frac{k}{2}}}{3}>2p^{k-1}$ holds. By rearranging we obtain the inequality 
$$(p-6)p^{k-1}-2\sqrt{p}p^{\frac{k-1}{2}}>0.$$

If $p-6>2\sqrt{p}$ then the inequality hold for every $k$ since $p^{k-1}\geq p^{\frac{k-1}{2}}$. Since $p-6-2\sqrt{p}$ is a quadratic polynomial in $\sqrt{p}$ it is easy to see that if $p>13$ then this condition is satisfied. If $p=7$, then we have to look at $(7-6)7^{k-1}-2\sqrt{7}p^{\frac{k-1}{2}}$. This will be positive if $7^{\frac{k-1}{2}}>2\sqrt{7}$ which happens if $k>2$. A similar calculation shows that if $p=11$ or $p=13$ then $k>1$ suffices. 
\end{proof}
\begin{remark}
We show two examples, when the above conditions are not satisfied and $f$ is not identically zero.

First consider the curve defined over $\F_3$:
$$C: y^2+2xy+2y+x=0.$$
It is easy to check that it has 2 affine points: $(0,0)$ and $(0,1)$. There $f(x)=x$ satisfies the condition and is nonzero.

Now we give an example of an elliptic curve over $\F_5$ for which $f$ can be chosen to be nonzero:
$$C: y^2=x^3+3x+1 .$$
$C$ has three affine points:$(0,1), (0,4)$ and $(1,0)$. Thus $f(x)=x$ is again a suitable choice. 

These do not exactly show the tightness of the bounds. However, if the conjecture in Remark \ref{two} is true, then we cannot expect much better examples. 
\end{remark} 

In this section we considered curves, where the degree of the curve is small compared to the characteristic of the field. It would be interesting to study curves where the degree is large compared to the characteristic of the field. Also, it would be even more interesting (and probably not hopeless) to give a complete characterization of those absolutely irreducible smooth curves $C$ for which there exists an additive function $f:\F_{p^k}\rightarrow \F_{p^k}$ such that $f$ is nonzero and $f(x)f(y)=0$ whenever $(x,y)\in C$. We leave these as open problems.

\subsection{Generalizations}

One can consider several related problems. The most natural continuation is to ask the following:

\begin{problem}
Let $f:\mathbb{F}_{p^k}\rightarrow\mathbb{F}_{p^k}$ be an additive function. Let $S$ be an absolutely irreducible smooth surface defined over $\mathbb{F}_{p^k}$. Assume that $f(x)f(y)f(z)=0$ whenever $(x,y,z)\in S$. Does this imply that $f$ is identically zero?
\end{problem}

The problem can be tackled using the method of Theorem \ref{main}. Assume that $f$ is not identically zero, thus there exists an $s\in\mathbb{F}_{p^k}$ such that $f(s)\neq 0$. Now by substituting $s$ into $z$ (or one of the other two variables) we obtain an absolutely irreducible curve (which is smooth) $C$ for which $f(x)f(y)=0$ if $(x,y)\in C$. If the degree of $S$ was $d$ then the degree of $C$ is at most $\left \lfloor{\frac{2d}{3}}\right \rfloor$ (by choosing the variable to be substituted in a suitable way). We have obtained the following:
\begin{proposition}
Let $S$ be an absolutely irreducible smooth surface of degree $d$. Let $f:\mathbb{F}_{p^k}\rightarrow\mathbb{F}_{p^k}$ be an additive function with the property that $f(x)f(y)f(z)=0$ whenever $(x,y,z)\in S$. Assume the degree of $S$ is fixed. Then if $p$ is large enough compared to $d$ then $f$ is identically zero. 
\end{proposition}
\begin{proof}
The proof follows from the previous discussion and Theorem \ref{main}. 
\end{proof}

Similar results can be obtained if we consider smooth varieties of arbitrary (bounded) dimension. Another natural question is the following:

\begin{problem}
Let $f:\mathbb{F}_{p^k}^2\rightarrow\mathbb{F}_{p^k}^2$ be an $\mathbb{F}_p$-linear function. Let $H$ be a hypersurface in $\mathbb{F}_{p^k}^4$. Assume that $f(x,y)f(u,v)=0$ whenever $(x,y,u,v)\in H$. Does this imply that $f$ is identically zero?
\end{problem}

This problem is harder than the previous one and it seems that it cannot be tackled with the above machinery. We leave it as an open problem here.

\section{The equation over $\mathbb{R}$}

We turn our attention to the curve $C: xy=1$. In other words we consider additive functions $f:\mathbb{R}\rightarrow\mathbb{R}$ for which $f(x)f(\frac{1}{x})=0$ (if $x\neq 0$). We show that surprisingly there exists a nonzero additive function with this property. We start with an observation:
\begin{proposition}\label{sub}
There exists an additive function $f:\mathbb{R\rightarrow\mathbb{R}}$ with the property that $f(x)f(\frac{1}{x})=0$ (where $x\neq 0$) if and only if there exists a $\Q$-subspace $U$ of the real numbers such that for every real number $x$ either $x$ or $\frac{1}{x}$ is in $U$. 
\end{proposition}
\begin{proof}
If $f$ is an additive function with this property, then the set $U=\{x\in\R| f(x)=0\}$ will suffice. Now consider the reverse direction. Assume such a $U$ is given. Then there exists a Hamel-basis of $U$ which can be extended to a Hamel-basis of $\R$. We set the value of those basis elements to zero which belong to $U$ and all the other values to 1. This shows that $f$ is nonzero and has the desired property.
\end{proof}

Now we give a construction of a nonzero additive function $f$ for which $f(x)f(\frac{1}{x})=0$. 
\begin{theorem}\label{main2}
There exists a non identically zero additive function $f:\mathbb{R}\rightarrow\mathbb{R}$ for which $f(x)f(\frac{1}{x})=0$ (whenever $x\neq 0$).
\end{theorem}
\begin{proof}
Let $\alpha$ be a transcendent number over $\Q$. Consider the field $K=\Q(\alpha)$, the field extension of the rational numbers by $\alpha$. The field $K$ is isomorphic to the field $\Q(t)$, the field of rational functions in one variable. Hence every element of $K$ is the quotient of two polynomials in $\alpha$ and we can define a valuation on $K$ in a similar fashion as in Example \ref{ex}. Every element of $K$ is a quotient of two polynomials in $\alpha$ thus the valuation is defined as the difference of the degree of the denominator and the numerator. This is well-defined as $\alpha$ is transcendental over $\Q$. This valuation is zero on $\Q$, but nontrivial on $K$. By Fact \ref{ext} the valuation on $K$ can be extended to a valuation $w$ on $\R$. Let $O$ be the valuation ring of $w$. Then $O$ is a $\Q$-subspace by Proposition \ref{val} and by the fact that the valuation $w$ is zero on $\Q$. Proposition \ref{val} implies that for every nonzero $x\in\R$ either $x$ or $\frac{1}{x}$ is in $O$. Now Proposition \ref{sub} implies the existence of a suitable $f$. 
\end{proof}
\begin{remark}
Instead of applying Fact \ref{ext}, Chevalley's theorem (Fact \ref{Chevalley}) can be applied directly as well. Indeed, let $\alpha$ be a transcendent real number and let $S$ be the set of those numbers in $\Q(\alpha)$ which are rational functions in $\alpha$ with the property that the degree of the denominator is at least as big as the degree of the numerator. Let $P$ be the prime ideal of $S$ consisting of those rational functions in $\alpha$ where the degree of the denominator is strictly larger than the degree of the numerator. Chevalley's theorem states that in this case there exists a valuation ring $O$ of $\mathbb{R}$ which contains $S$. Thus $O$ satisfies the conditions of Proposition \ref{sub} since $\Q$ is contained in $S$ 
\end{remark}
\begin{remark}
The valuation described in the proof of Theorem \ref{main2} can be constructed explicitly as well. Take a transcendence basis $B$ of $\mathbb{R}$ over $\Q$. Now $\mathbb{R}$ is an algebraic (but not finite) extension of the purely transcendental extension $\Q(B)$. A valuation similar to the degree valuation described in the proof of Theorem \ref{main2} (and in Example \ref{ex}) can be defined on $\Q(B)$ as on $\Q(\alpha)$ (as the elements of $\Q(B)$ are rational functions in infinitely many variables). Finally, there exists a standard procedure to extend valuations to algebraic extensions (see \cite{Engler}). 
\end{remark}

The proof also implies that instead of $\R$ any field $F$ which is transcendental over $\Q$ has a nonzero additive function $f: F\rightarrow F$ satisfying the property that $f(x)f(\frac{1}{x})=0$. Similarly one can show that every field which is transcendental over $\mathbb{F}_p$ has an additive function with the same property. On the other hand Theorem \ref{main} shows that if $K$ is a finite field of odd characteristic then an additive function $f$ satisfying $f(x)f(\frac{1}{x})=0$ must be identically 0. This implies that if $K$ is the algebraic closure of a finite field of odd characteristic then such an additive $f$ must be identically zero on $K$ (as finite extensions of finite fields are finite fields also). Thus in the odd characteristic case (let $p$ be the characteristic) one has that an additive function $f$ with the property that $f(x)f(\frac{1}{x})=0$ must be identically zero on a field $K$ if and only if $K$ is algebraic over $\mathbb{F}_p$. The next proposition (due to G\'abor Ivanyos) shows that the same is true when the characteristic of $K$ is zero. 

\begin{proposition}\label{IG}
Let $K$ be an algebraic number field and let $f: K\rightarrow K$ be an additive function such that for every $x\neq 0$ one has that $f(x)f(\frac{1}{x})=0$. Then f is identically zero. 
\end{proposition}
\begin{proof}
Let $K$ be a number field of degree $n$ over $\Q$. Assume that the kernel of $f$ (i.e., the $\Q$-subspace of those $x\in K$ for which $f(x)=0$) has dimension $l<n$ over $\Q$. We use induction on $n$. Note that the statement is trivially satisfied if $K=\Q$. Let $x\in K$ be an arbitrary element. We define the set $H_x$ (where $x\in K$) in the following way:
$$H_x=\{k\in \mathbb{N}| f(x^k)=0\}.$$
Thus $H_x$ is a subset of the natural numbers. By Szemerédi's theorem \cite{Szemeredi} either $H_x$ or its complement in $\mathbb{N}$ contains an arithmetic progression of length $l+1$ where every element is smaller than $g(l)$ a function depending only on $l$ (meaning it does not depend on $x$). If $H_x$ contains such an arithmetic progression, then there exists $k,d\in\mathbb{N}$ such that $x^k,x^{k+d},\dots, x^{k+ld}$ lie in the kernel of $f$. Since the dimension of the kernel is $l$, the elements $x^k,x^{k+d},\dots, x^{k+ld}$ are linearly dependent over $\Q$, i.e., there exist $a_0,\dots,a_l$ rational numbers such that $a_0x^k+\dots+a_lx^{k+ld}=0$. Dividing by $x^k$ leads to 
$$a_0+a_1x^d+\dots+a_lx^{ld}=0.$$
Thus $x^d$ lies in a proper subfield of $K$. If the complement of $H_x$ contains the arithmetic progression then $(\frac{1}{x})^d$ lies in a proper subfield of $K$, thus we can conclude that for all $x$ we have that $x^d$ lies in a proper subfield of $K$. The induction hypothesis implies that there exists a universal $D$, such that $f(x^D)=0$ (we may choose $D=g(l)!$) for every $x\in K$. 

The map $f$ can be considered as a linear map from $\Q^n$ to $\Q^n$ (by identifying $K$ with $\Q^n$ as a vector space over $\Q$). Consider the tensor product $\A=K\otimes_{\Q}\mathbb{C}$ which is isomorphic to $\mathbb{C}^n$ and contains $K$ as a subfield. The map $f$ extends to $\A$. Then $g: \A\rightarrow \A$ defined as $g(x)=f(x^D)$ is a polynomial function which is not identically zero, since the map $x\mapsto x^D$ is surjective and $f$ is not identically zero. On the other hand it must vanish on $K$ which cannot happen as a polynomial function with rational coefficients which vanishes on $\Q^n$ must be identically zero (meaning that each coordinate function must be the zero polynomial). This yields a contradiction.   
\end{proof}
\begin{remark}
It would be interesting to find a proof for Proposition \ref{IG} which does not rely on Szemer\'edi's theorem on arithmetic progressions (maybe just using the fact that for every $x\in K$ there exists a positive integer $k$ for which $f(x^k)=0$).
\end{remark}
\begin{corollary}
Let $\overline{\mathbb{Q}}$ be the algebraic closure of the rational numbers. Let $f: \overline{\mathbb{Q}}\rightarrow \overline{\mathbb{Q}}$ be an additive function such that for every $x\neq 0$ one has that $f(x)f(\frac{1}{x})=0$. Then f is identically zero. 
\end{corollary}
\begin{proof}
Assume that there exists such an $f$ which is not identically zero. Thus there exists an $\alpha\in\overline{\mathbb{Q}}$ such that $f(\alpha)\neq 0$. However, every element in $\overline{\mathbb{Q}}$ is algebraic over $\Q$, thus $K=\Q(\alpha)$ is an algebraic number field. Restricting $f$ to $K$ would yield an additive function $f'$ on $K$ which contradicts Proposition \ref{IG}.
\end{proof}

 Now we provide some more curves $C$ for which there exists a non-continuous additive function $h$ such that $h(x)h(y)=0$ whenever $(x,y)$ lies on the curve $C$. 

\begin{proposition}\label{ext2}
Let $C$ be an affine curve given by the rational parametrization $C: (f(t),g(\frac{1}{t}))$, where $f,g\in\Q[x]$ are polynomials with rational coefficients. Then there exists a nonzero additive function $h:\mathbb{R}\rightarrow\mathbb{R}$ such that $h(x)h(y)=0$, whenever $(x,y)\in C$. 
\end{proposition}
\begin{proof}
Let $v$ be the valuation defined in Theorem \ref{main2}. Let 
$$f(t)=\sum_{i=0}^n a_it^i, ~g(t)=\sum_{i=0}^m a_it^{-i}.$$
Let $O_v$ be the valuation ring of the valuation $v$. We set $h$ to be identically zero on $O_v$ and extend it to $\mathbb{R}$ in a fashion that it is not identically zero on $\mathbb{R}$. If $t\in O_v$ then so is $f(t)$ as $O_v$ is a ring which contains $\Q$. If $t$ is not contained in $O_v$ then $\frac{1}{t}\in O_v$ thus $g(\frac{1}{t})\in O_v$ hence $h(x)h(y)=0$ whenever $(x,y)\in C$. 
\end{proof}
\begin{remark}
A natural question is whether Proposition \ref{ext2} holds if $f(x),g(x)\in\Q(x)$ are rational functions where the degree of $f$ is at least zero and the degree of $g$ is at most zero. The answer is negative, as the unit circle ($C: x^2+y^2=1$) admits such a parametrization $C: (\frac{t^2-1}{t^2+1},\frac{2t}{t^2+1})$ but in \cite{KRS} it is proven that such an $h$ must be identically zero. 
\end{remark}
\begin{remark}
Actually the proof of Proposition \ref{ext2} shows that the following stronger statement is true. Let $\cal{H}$ be the family of all those curves which can be parametrized as $(f(t),g(\frac{1}{t}))$. Then there exists a non-continuous additive function $h$ such that $h(x)h(y)=0$ whenever $(x,y)$ lies on some curve $C\in\cal{H}$. 
\end{remark}

The results of the paper motivate certain questions. Proposition \ref{ext2} provides a family of curves for which there exist non-continuous additive functions $h:\mathbb{R}\rightarrow\mathbb{R}$, where $h(x)h(y)=0$ if $(x,y)$ lies on the curve. This family include curves of arbitrarily large degree, however they are all rational curves, thus of genus 0. To our knowledge, for genus 1 curves (i.e., elliptic curves) nothing is known. This means that we have no example for an elliptic curve $C$ for which such a nonzero $h$ exists, or which implies that $h$ must be identically zero. Another direction for future investigations could be to study the problem for quadratic functions instead of additive ones. Some investigations have been made in this area, see \cite{BFK}. 
\paragraph*{Acknowledgement}
The author would like to thank Zolt\'an Boros and Lajos R\'onyai for their useful remarks and their constant support and also G\'abor Ivanyos for helpful suggestions and providing the proof of Proposition \ref{IG}. Research supported by the Hungarian National Research, Development and Innovation Office - NKFIH.

\end{document}